\documentclass[12pt,letterpaper,reqno]{amsart}

\usepackage{amsthm}
\usepackage{mathrsfs}
\usepackage{amsfonts,amssymb,amsmath}
\input amssym.def 
\input amssym.tex

\usepackage[dvips]{graphicx}
\usepackage{hyperref}

\newtheorem{con}{\bf Conjecture}

\newtheorem{lem}{\bf Lemma}

\newtheorem{thm}{\bf Theorem}
\newtheorem{corr}{\bf Corollary}

\newtheorem{quest}{\bf Question}

\newcommand\gc{\gcd}
\newcommand\be{\begin{eqnarray*}}
\newcommand\ee{\end{eqnarray*}}
\newcommand\beq{\begin{equation}}
\newcommand\eeq{\end{equation}}

\newcommand\ben{\begin{eqnarray}}
\newcommand\een{\end{eqnarray}}
\newcommand\ord{\mathrm{ord}}

\begin{document}

\title[Product sets cannot contain long arithmetic progressions] 
{Product sets cannot contain long arithmetic progressions}

\author{Dmitrii Zhelezov}
\thanks{Department of Mathematical Sciences, 
Chalmers University Of Technology \and
Department of Mathematical Sciences,  University of Gothenburg} 
\address{Department of Mathematical Sciences,
Chalmers University Of Technology, 41296 Gothenburg, Sweden} 
\address{Department of Mathematical Sciences, University of Gothenburg,
41296 Gothenburg, Sweden} \email{zhelezov@chalmers.se}


\date{\today}

\begin{abstract}
  Let $B$ be a set of natural numbers of size $n$. We prove that the length of the longest arithmetic progression contained in the product set $B.B = \{bb'| \, b, b' \in B\}$ cannot be greater than $O(\frac{n\log^2 n}{\log \log n})$ and present an example of a product set containing an arithmetic progression of length $\Omega(n \log n)$. For sets of complex numbers we obtain the upper bound $O(n^{3/2})$.
\end{abstract}

\maketitle
\section{Introduction}
  Sum-product estimates are among the most important questions in modern additive combinatorics. In general, one wants to show that if there is enough additive structure in a set $A$ (for example if it has small doubling constant $\frac{|A+A|}{A}$), then the {\it product set} $A.A = \{aa'|a, a' \in A\}$  is large. The most famous conjecture in this area was posed by Erd\H{o}s and Szemer\'edi \cite{ES}, which says that for any set $A$ of complex numbers holds
 $$
 \max(|A.A|,|A+A|) \geq c|A|^{2-\epsilon},
 $$
 for arbitrary $\epsilon > 0$ and some $c > 0$ that may depend on $\epsilon$. The state of the art exponent $4/3-o(1)$ was obtained by Solymosi in a very elegant way \cite{S}. It is worth noting that each new bound for the exponent required a substantial new idea and attracted considerable attention from experts in the field.
 
 In this note we investigate a different sort of relationship between the additive structure and the size of a product set. Namely, we show that a product set cannot contain extremely long arithmetic progressions. The result is the following.
\begin{thm} \label{thm:main}
  Suppose that $B$ is a set of $n$ natural numbers. Then the longest arithmetic progression in $B.B$ has length at most $O(\frac{n\log^2 n}{\log \log n})$.
\end{thm} 
A lower bound is provided by the following theorem. 
\begin{thm} \label{thm:lowerbound}
  Given a integer $n > 0$ there is a set $B$ of $n$ natural numbers such that $B.B$ contains an arithmetic progression of length $\Omega(n \log n)$.
\end{thm} 

In the fourth section of this note we will extend Theorem \ref{thm:main} to sets of complex numbers, but with a considerably weaker bound $O(n^{3/2})$. 
 
\section{Notation}
The following standard notation will be used in this paper: 
\begin{enumerate}
\item $f(n) = O(g(n))$ means that $\limsup_{n \rightarrow \infty} \frac{f(n)}{g(n)} < \infty$.
\item $f(n) = \Omega(g(n))$ means that $g(n) = O(f(n))$.
\item Let $H$ be a fixed graph. Then $\mathrm{ex}(n, H)$ denotes 
the maximal number of edges among all graphs with $n$ vertices which does not contain $H$ as a subgraph. In particular, $\mathrm{ex}(n, C_{k})$ denotes the maximal number of edges a graph with $n$ vertices avoiding cycles of length $k$ can possibly have.
\item Let $p$ be a prime, then $d = \ord_p(n)$ denotes the maximal power of $p$ such that $p^d \mid n$.

\end{enumerate}

\section{Main Result} 
   Let $A = \{r + di\}, i = 0, \ldots, N$ be an arithmetic progression in the product set $B.B$ of a set $B$ of size $n$. We start with the observation that 
by taking absolute values of $B$ the longest arithmetic progression in $B.B$ can be shortened by a factor at most two, 
so we may assume that all elements in $B$ are positive. 
 
  We proceed with the following technical lemma.  
\begin{lem} \label{lem:reduct}
  We may assume that $A = \{D(r' + d'i)\}$ for some $D > 0$ such that $\mathrm{gcd}(d', Dr') = 1$.
\end{lem}  
\begin{proof}
	Let $p$ be a prime such that, $\ord_p(d) > \ord_p(r)$. If there is no such $p$ then $D = \mathrm{gcd}(r, d), d' = d/D, r' = r/D$ provides the desired factorization. 
	If $k' = \ord_p(r) = 1$ then every number in $A$ is a product $b_ib_j$ such that $p \mid b_i$ but $p \nmid b_j$ and thus we can reduce $B$ to 
	$$B' = \{b_i |\,\, b_i \in B, p \nmid  b_i \} \cup \{\frac{b_i}{p} |\,\, b_i \in B, p \mid b_i  \}$$ 
	and iterate the lemma again.
	
	So, now we assume that $k = \ord_p(d) > k' > 1$. We divide $B$ into three sets $B_1, B_2, B_3$ such that $b_i \in B_1$ if  $p \nmid b_i$, $b_i \in B_2$ if $0 < \ord_p(b_i) < k'$ and finally $b_i \in B_3$ if $p^{k'} \mid b_i$. Since $\ord_p(d) > k'$ for every $a \in A$ we have $\ord_p(a) = k'$ and $a$ can be either a product of two numbers from $B_2$ or a product $b_1b_3$ where $b_1 \in B_1$ and $b_3 \in B_3$. Thus, we can reduce $B$ to 
	$$B' = \{b_i |\,\, b_i \in B_1 \} \cup \{ \frac{b_i}{p} |\,\, b_i \in B_2 \} \cup \{\frac{b_i}{p^2} |\,\, b_i \in B_3 \}$$ 
	such that $B'.B'$ contains an arithmetic progression $A/p^2$ of the same length as $A$, and then iterate the lemma.

\end{proof} 

From now on we will assume the factorization $A = \{D(r + di)\}$, such that $\gc(Dr, d) = 1$. By $N$ we will always denote the length of $A$ and $a_k = D(r+dk)$ will be the $k$th element of $A$ (if not stated explicitly). 

\begin{lem} \label{gcdbound}
  For  $i \neq j$ 
  $$\mathrm{gcd}(a_i, a_j) \leq DN.$$
\end{lem}  
\begin{proof}
	For $i > j$ we have
	\be
		\gc(a_i,a_j) &=& \gc(a_i-a_j, a_j) = \gc(Dd(i-j), D(di+r)) \\
								 &=& D\gc(d(i-j), di+r) = D\gc(i-j,di+r) \leq DN.
	\ee
	The last equality follows from $\gc(d, Dr) = 1$.
\end{proof}

 Let us fix a single pair $b_i, b_j \in B$ for each $a \in A$ such that $b_ib_j = a$ and make a graph $G$ with $b \in B$ as vertices, such that for every $a \in A$ there is a unique edge between $b_i$ and $b_j$ which has been previously fixed for such $a$ (for each edge we can simply take the first representation of $a$ in lexicographical order). We will have $n = |B| = |V(G)|$ and $N = |A| = |E(G)|$. 
  It turns out that our further analysis significantly simplifies if $G$ is simple (without loops) and bipartite. However, we can always achieve this sacrificing just a constant factor by simply taking two copies of $B$, say $B_1$ and $B_2$ that are going to be the color classes of $G$, such that for each edge $e = \langle b_i, b_j \rangle \in G, \, i \leq j$ we place an edge between $b_i \in B_1$ and $b_j \in B_2$, so the resulting graph is bipartite and simple.        
 
 As we will see from our example, which provides a lower bound $N = \Omega(n\log n)$, it is safe to assume $N > 2n$, a very weak yet convenient bound, as it guarantees, for example, that $G$ contains a cycle.


  

\begin{lem} \label{lm:dbound}
	If $G$ contains an even cycle of size $2k$, then $r \leq N^k$ and $d \leq N^{k}$. 
\end{lem}
\begin{proof}
  Let $C = b_1b_1...b_{2k}$ be a simple cycle in $G$ of length $2k \leq n$, so $b_ib_{i+1} \in A, i=1,..,2k$ (hereafter we assume addition of indices modulo $2k$).  By simple algebra we have 
  
\beq \label{eq:chain}
  b_{2k}b_1 = \frac{b_1b_2}{b_2b_3}\frac{b_3b_4}{b_4b_5}\cdots \frac{b_{2k-3}b_{2k-2}}{b_{2k-2}b_{2k-1}}b_{2k-1}b_{2k},
\eeq
and since for each $i$ there is some $j$ such that $b_ib_{i+1} = D(r+j_id)$ we can rewrite (\ref{eq:chain}) as 
\beq \label{eq:chain2}
  \prod^k_{i = 1}{(r+j_{2i}d)}   = \prod^k_{i=1}{(r+j_{2i-1}d)},
\eeq
where all $j_i$ are distinct (since for ever $a \in A$ we have chosen only a single representation). Expanding the brackets, we obtain the equation
\beq \label{eq:poly}
  c_0r^k + c_1r^{k-1}d + ... + c_{k-1}rd^{k-1} + c_kd^k = 0
\eeq
for integer coefficients $c_i$ which depend only on indexes $j$. First, let us note that it cannot happen that all $c_i  = 0$ since otherwise (\ref{eq:chain2}) holds for \emph{any} $r, d$ which contradicts the fact that all $j$s are distinct. Let $l$ and $m$ be the smallest and largest indices such that $c_l, c_m \neq 0$ respectively. Obviously, $l < m$ and dividing (\ref{eq:poly}) by $r^ld^{k-m}$ we arrive at
\beq \label{eq:polyred}
  c_lr^{m-l} + ... + c_md^{m-l} = 0.
\eeq
Since $r$ and $d$ are coprime, $r | c_m$ and $d | c_l$ (all the terms in the middle are divisible by $rd$), and the claim of the lemma follows if the bound $c_i \leq N^k$ holds for all coefficients. But on the other hand, $c_t$ is a sum of $2\binom {k} {t}$ $t$-fold products of $j$s. Since each index $j$ is less than $N$, for $t \leq k/2$ we have
$$
 c_t \leq 2k^t N^{t} < n^t N^{t} < N^k,
$$
and analogously, for $t \geq k/2$
$$
 c_t \leq 2k^{k-t} N^{t} < n^{k-t} N^{t} < N^k.
$$
Here we used the trivial bound $2k \leq n$.

\end{proof}

\begin{lem} \label{lem:main}
	If $d < N^{k}$, $r < N^{k}$, $3^{k} < N/9$ then $N \leq 36kn \log n$ for sufficiently large $n$.
\end{lem}
\begin{proof}
 Suppose for contradiction that $N > 36(k+1)n\log n$. Let $p_1, ..., p_K$ be primes such that $N/3 < p_i < N/2$ and $p_i \nmid d$. By the Prime Number Theorem there are more than $N/6\log N > 3(k+1)n$ primes in $[N/3, N/2]$ (for $N$ large enough) and at most $k$ of them may divide $d$ (since $d < N^k$ and $3^{k+1} < N$\footnote{This is the only place where we use the technical bound $3^{k} < N/9$, but as we will see later, this restriction does not affect the final bound as $k$ is going to be $o(log n)$}), so $K > 3(k+1)n$. 
 
 Recall the graph $G$ with $b \in B$ as vertices and edges that correspond to the relation $b_ib_j \in A$, with each representation of $a \in A$ being unique. Let us call an edge of $G$ \emph{regular}, if 
 $$
 	\gc(\frac{b_ib_j}{D}, p_1...p_K) = 1,
 $$
 or, in words, if $b_ib_j$ does not have any additional power of the aforementioned $p_1, ..., p_K$ in its prime decomposition. Otherwise, if $\ord_p(b_ib_j) > \ord_p(D)$ let us call an edge $(b_i, b_j)$  $p$-\emph{irregular}. Further, by saying just "an irregular edge", we mean an edge that is $p$-irregular for at least one $p \in \{p_1,...,p_K\}$. Note, that it can be irregular for some primes, but regular with respect to others.

  Let $p \in P_K = \{p_1,...,p_K\}$. Since $p \nmid d$, $dj$ covers the full system of residues modulo $p$ when $j$ goes from $0$ to $N$. Hence, since $p \in [N/3, N/2]$, there are either two or three indices $j$ such that $p \mid dj + r$, and thus two or three $p$-irregular edges in $G$.
  
  By the pigeonhole principle, we can pick a set $S$ of at least $n+1$ distinct irregular edges, such that for every $p \in P_K$ there is at most one $p$-irregular edge in $S$. Indeed, every element in $A$ can have at most $k+1$ divisors in $P_K$ (due to the bounds $d < N^{k}, r < N^{k}$ we have $r + id < N^{k+1}$ for $0 \leq i \leq n$). On the other hand, for every $p \in P_K$ there are at most three elements in $A$ it divides.  
  
  
 
 The next step is to clean up our original graph $G$ by removing all edges except that are not in $S$. We will refer to the resulting graph as $G'$. Of course, it is simple and bipartite as was $G$. Now we claim that it contains no cycles. Indeed, let $e_p$ be a (unique) $p$-irregular edge in $G'$ and $e_p=a_1a_2...a_{2l}$ be a cycle it lies on (of course, here indices of $a$'s indicate just the ordering in the cycle, not in $A$). Note, that now we write the cycle as a set of edges rather than vertices, meaning that $a_i \in A$ and each $a_i$ is a product of two consecutive vertices of the cycle. Thus, arguing exactly as in Lemma \ref{lm:dbound} it is easy to see that
 $$
 	\prod_{i \text{ is odd}} a_i = \prod_{i \text{ is even}} a_i.
 $$
But this cannot happen. Indeed, for each $a_i \neq e_p = a_1$ we have $\ord_p(a_i) = \ord_p(D)$ since $e_p$ is the only $p$-irregular edge in $G'$, and the $p$-order of the RHS is strictly less than of the LHS. Thus, $G'$ cannot contain more than $n$ edges. Contradiction.

\end{proof}

Putting it all together, we obtain the main result of this note.

\begin{proof} \textbf{[of Theorem 1]}
   If $G$ does not contain even cycles of length up to $2k$ the result of Bondy and Simonovits from extremal combinatorics \cite{BS} gives 
   \beq \label{eq:cyclebound}
   N \leq \mathrm{ex}(n, C_{2k}) < 100kn^{1+1/k}.
   \eeq
   But otherwise Lemmas \ref{lm:dbound} and \ref{lem:main} apply and we obtain $N \ll(k+1)n\log n$, so finally we have
   $$
   	  N \leq O(\max\{kn^{1+1/k}, kn\log n\}).
   $$
   This can be optimized by taking $k = \log n/\log \log n$ which gives the desired bound $N = O(n\log^2 n/\log \log n)$.
\end{proof}

Now we present a construction for the lower bound of Theorem 2.
\begin{proof} \textbf{[of Theorem 2]}
   Consider a set $B$ which consists of all natural numbers from $1$ to $n$ plus all primes in the interval $[n, \lfloor n\log n \rfloor]$. By the Prime Number theorem, $|B| \leq 2n$ for large $n$ and $B.B$ contains all natural numbers in the interval $[1, \lfloor n\log n \rfloor]$ which is an arithmetic progression of size $\Omega(n\log n)$. 
   
   Indeed, suppose $x \in [n, \lfloor n\log n \rfloor]$. If the maximal prime $p$ that divides $x$ is greater than $\log n$ than $x/p \leq n$ and $x = p\cdot\frac{x}{p}$ is clearly in $B.B$, since all primes in the interval $[1, \lfloor n\log n \rfloor]$ are in $B$. Otherwise, run the following algorithm. Let $p_1$ be an arbitrary prime divisor of $x$ and assign $d_1 = p_1, d_2 = x/p_1$. 
Then choose the smallest prime divisor $p'$ of $d_2$, assign $d_1 := d_1p', d_2 := d_2/p'$ and iterate this procedure until $d_2 = 1$. If there is a moment when both $d_1, d_2 \leq n$ then of course $x \in B.B$ and we are done. Otherwise, at some step $d_1 < n, d_2 > n$, but $d_1p' > n, d_2/p' < n$. But since every prime divisor of $x$ is less than $\log n$ we have 
$$
x = d_1d_2 \geq \frac{n^2}{\log n},
$$
 which contradicts that $x \in [n, \lfloor n\log n \rfloor]$.
   
\end{proof}

\section{The case of complex numbers}

\begin{thm} \label{thm:reals}
	 Suppose that $B$ is a set of $n$ complex numbers. Then the longest arithmetic progression in $B.B$ has length at most $O(n^{3/2})$.
\end{thm}

Our strategy will be to show that if $B.B$ contains an arithmetic progression $A$ of size $\Omega(n^{3/2})$ then in fact one can take a new set $B'$ of only rational numbers, perhaps twice as big as the original set $B$, such that $B'.B'$ contains a progression of the same length. Unfortunately, we can prove that such a reduction exists only if the arithmetic progression $A$ in the original set has length at least $\Omega(n^{3/2})$, so the resulting bound is much weaker than what Theorem \ref{thm:main} gives for sets of natural numbers.

So let $A = \{r + di\}$ be an arithmetic progression of length $N$ in $B.B$. The first step is to scale $A$ by simply dividing each element in $B$ by $\sqrt{d}$, and from now on we will assume that $A = \{r + i\}$. 

%
%
Recall the graph $G$ which provides a one-to-one correspondence between elements of $A$ and its edges, namely an edge $e_a = \langle b_i, b_j \rangle$ corresponds to the element $a=b_ib_j$.

\begin{lem} \label{lem:ratA}
	If $G$ contains a $4$-cycle then $r$ is rational and so are all elements of $A = \{r + i\}$.
\end{lem}
\begin{proof}
		Let $\langle b_1b_2b_3b_4 \rangle$ be a $4$-cycle in $G$. Then both $b_1(b_2 - b_4)$ and $b_3(b_2 - b_4)$ are non-zero integers as they are differences of two distinct elements of $A$. Thus, $b_1/b_3$ is rational and so is $q = b_1b_2/b_2b_3 \neq 1$. On the other hand, writing $b_1b_2 = r + i_1$ and $b_2b_3 = r + i_2$, we have
		$$
		\frac{r+i_1}{r+i_2} = q,
		$$ 
so 
$$
	 r = \frac{i_1 - qi_2}{q - 1}
$$
is rational since $i_1, i_2$ are integers.
\end{proof}
     
\begin{corr} \label{cor:ratA}
	If $A = \{r + i\}$ is contained in a product set $B.B$ with $|B|=n$ and $|A| > n^{3/2}$ then it consists of rational numbers.
\end{corr}
\begin{proof}
	The claim follows from the well-known fact that a graph with more than $n^{3/2}$ edges contains a $4$-cycle\footnote{In fact, $\mathrm{ex}(n, C_4) \leq \frac{n}{4}(1+\sqrt{4n-3})$, see \cite{R}.} together with Lemma \ref{lem:ratA}.
\end{proof}

   While the condition that all elements in $A$ are rational is strong, it still does not guarantee that elements in $B$ are rational as well, so some additional tweaks are needed in order to invoke Theorem \ref{thm:main}. We will construct a slightly different set $B'$ of only rational numbers such that $B'.B'$ contains $A$. Our main observation is the following.
   
\begin{lem} \label{lem:ratpath}   
   Assume $A$ consists of rational numbers. Then if $b_i$ and $b_j$ are connected in $G$ by a path of even length, the quotient $b_i/b_j$ is rational. If they are connected by a path of odd length, the product $b_ib_j$ is rational. 
\end{lem}   
\begin{proof}
   Indeed, if there is a path $L = \langle b_i, b_{i+1}...b_{i+2k+1}=b_j \rangle$ of even length we have
\beq \label{eq:rationallen}  
  \frac{b_i}{b_j} = \frac{(b_ib_{i+1})(b_{i+2}b_{i+3})...(b_{i+2k-1}b_{i+2k})}{(b_{i+1}b_{i+2})...(b_{i+2k}b_{i+2k+1})},   
\eeq
which is rational. In exactly the same way the second claim of the lemma follows.
\end{proof}   
   
 Our next step is to make elements in $B$ rational while preserving the property that $A$ is contained in $B.B$. Remember, that from the very beginning we assume our graph $G$ simple bipartite (which one can always do WLOG).

\begin{lem} \label{lem:struct}   
   Let $A$ be a subset of $B.B$ consisting of only rational numbers and the corresponding incidence graph $G$ is bipartite.  Then there is a set of rational numbers $B'$ of size $|B|$ such that $A \subset B'.B'$.
   
\end{lem}   
\begin{proof}   
   Let $K_1, K_2, \ldots, K_l$ be the connected components of the bipartite graph $G$. We will treat them separately one by one. So let $K$ be one of the components. 
 As $K$ does not contain odd cycles, we can color its edges in black and white such that there are edges only between white and black vertices. 
 
  By Lemma \ref{lem:ratpath} the quotient $b_i/b_j$ is rational for the vertices of the same color, and so is the product of any two vertices of different color.
Thus, we can take an arbitrary white element $b_w$ from $K$ and modify our set $B$ as follows:
\begin{itemize}
	\item For all white $b \in K$ set $b := b/b_w$
	\item For all black $b \in K$ set $b := bb_w$. 
\end{itemize}
  As $K$ is bipartite, this procedure will keep the set $A$ unchanged. On the other hand, it makes all the elements in $K$ rational. 
  
  Iterating the procedure above for all components, we finally obtain the set $B'$ with the desired properties.
      
\end{proof}

\begin{proof} \textbf{[of Theorem \ref{thm:reals}]} Now the theorem follows as an immediate corollary of Corollary \ref{cor:ratA} and Theorem \ref{thm:main} since multiplying our new set $B'$ by a sufficiently composite number we obtain a set of integers whose product set contains an arithmetic progression of the same length. It remains to note that by taking absolute values of $B'$ the longest arithmetic progression in $B'.B'$ can be shortened by a factor of at most two.
\end{proof}

\section{Discussion}
  The motivation for asking how long an arithmetic progression in a product set can be stems from the question asked by Hegarty \cite{H}.
  
  \begin{quest}
  Let $B$ be a set of $n$ integers and let $A$ be a strictly convex (concave) subset of $B+B$. Must $|A|=o(n^2)$? 
  \end{quest}
  
  Recall that a sequence of numbers $A = \{a_1, ..., a_n \}$ is called strictly convex (concave) if the consecutive differences $a_i - a_{i-1}$ are strictly increasing (decreasing). 
  
  It is not difficult to see that it does not matter whether the numbers in Question 1 are reals or integers. Now suppose that $B = \{\log b_i'\}$ for some $b_i'$, so $B + B = \{ \log (b_i'b_j')\}$. If $B'.B' = \{b_i'b_j'\}$ contains a long arithmetic progression, we immediately obtain a convex set of the same size in $B+B$. If we assume that $b_i'$ are natural numbers then Theorem 1 shows that the longest convex set we can possibly get in this way is of size $O(n^{1+o(1)})$. Apart from Hegarty's original inquiry, we now ask the following question that might be simpler.


 
\begin{quest}
  Can one construct an example of a set of size $n$ such that the sumset $B+B$ contains a convex (concave) set of size $n^{1+\delta}$ for some $\delta > 0$ and arbitrarily large $n$? 
\end{quest}  

\noindent \textbf{Remark.}   
  Erd\H{o}s and Newman in \cite{EN} gave an example of a set $B$ of size $\frac{n}{\log^Mn}$ such that $B+B$ covers $\{1, 2^2,..., n^2\}$ for arbitrary $M > 0$, which is better than our construction above, but still this lower bound is very weak. 
\\
\noindent \textbf{Remark.}   
 Erd\H{o}s and Pomerance in \cite{EP} asked if it is true that for a large enough $c$, every interval of length $cn$ contains a number divisible by precisely one prime in $(n/2, n]$? While the question remains open, a positive answer would give an essentially sharp upper bound $O(n \log n)$ for Theorem 1.
\\  

An obvious direction of research is to match the bound for the case of complex numbers to the one of Theorem \ref{thm:main}. Moreover, we believe that the lower bound $O(n\log n)$ is sharp for Theorem \ref{thm:main} and perhaps for Theorem \ref{thm:reals} as well.

Another interesting twist is to ask the question of the current note for subsets of finite fields $\mathbb{F}_p$. By a recent result of Grosu \cite{G}, 
the bound of Theorem \ref{thm:reals} translates to subsets $B \subset \mathbb{F}_p$ of size $O(\log \log \log p)$. While there are sets $B$ of size $O(\sqrt{p})$ such that $B.B$ covers the whole field $\mathbb{F}_p$ and thus contains an AP of size $\Omega(|B|^2)$, we conjecture that for smaller sets the bound $|B|^{1+o(1)}$ holds.
\begin{con}
	There is an absolute constant $c > 0$ such that for all $B \subset \mathbb{F}_p, |B| < c\sqrt{p}$ the product set $B.B$ cannot contain an arithmetic progression of size greater than $|B|^{1+o(1)}$. Here we assume $p$ and $|B|$ are large.
\end{con}

  Finally, a lot of related questions arise if we continue the general idea of asking how large a set with additive structure can be if it is contained in a product set? For example, instead of arithmetic progressions one may ask about generalized arithmetic progressions or just sumsets of an arbitrary set.

\section{Acknowledgements}
  I am very grateful to my supervisor Professor Peter Hegarty for helpful discussions and constant support. I would also like to thank Boris Bukh for comments during the poster session at the Erd\H{o}s 100 conference in Budapest which helped to improve the exposition. Finally, I thank the anonymous referee for valuable comments and especially for pointing out that applying the result of Bondy and Simonovits \cite{BS} to our case actually gives a better bound than a more recent result of Lam and Verstra{\"e}te \cite{LV}.


\begin{thebibliography}{99}
	\bibitem{BS} J. A. Bondy and M. Simonovits, \textit{Cycles of even length in graphs}, J. Combinatorial Theory Ser. B, 16 (1974), 97--105.
	\bibitem{EN} P. Erd\H{o}s and D. J. Newman, \textit{Bases for sets of integers}, J. Number Theory 9 (1977), 420--425.
	\bibitem{EP} P. Erd\H{o}s and C. Pomerance, \textit{Matching the natural numbers up to $n$ with distinct multiples in another interval}, Nederl. Akad. Wetensch. Proc. Ser. A 83 (1980), 147�-161.
	\bibitem{ES} P. Erd\H{o}s and E. Szemer\'edi, \textit{Sums and products of integers}, Adv. Stu. P. M. (1983), 213--218.
	\bibitem{G} C. Grosu, \textit{$\mathbb{F}_p$ is locally like $\mathbb{C}$}, arXiv:1303.2363.
	\bibitem{H} {P. Hegarty}, {http://mathoverflow.net/questions/106817/convex-subsets-of-sumsets}
	\bibitem{LV} T. Lam  and J. Verstra{\"e}te, \textit{A note on graphs without short even cycles}, Electron. J. Combin. 12.1 (2005)
	\bibitem{R} I. Reiman, \textit{\"Uber ein Problem von K. Zarankiewicz}, Acta. Math. Acad. Sci. Hungar. 9 (1958), 269�-273. 
	\bibitem{S}  J. Solymosi, \textit{Bounding multiplicative energy by the sumset}, Adv. Math. 222 (2009), 402--408.
	
\end{thebibliography}
\end{document}